\documentclass[12pt]{amsart}
\usepackage[]{amsmath, amsthm, amsfonts, amssymb, epsfig}
\newcommand\C{{\mathbb C}}

\newcommand\RR{{\mathcal R}}
\newcommand\KK{\mbox{\boldmath$\mathcal K$}}
\newcommand\RS{{\mathcal R}_s}
\newcommand\blambda{\mbox{\boldmath$\lambda$}}
\newcommand\dee{\partial}
\newcommand\Om{\Omega}
\newcommand\Obar{\overline{\Omega}}
\newcommand\Ot{\widetilde\Omega}
\newcommand\Ohat{\widehat\Omega}

\renewcommand\phi{\varphi}

\numberwithin{equation}{section}

\begin{document}

\title[Algebraic geometry and the Bergman kernel]
{Real algebraic geometry of real algebraic Jordan curves in
the plane and the Bergman kernel}
\author[S.~Bell]
{Steven R.~Bell}

\address[]{Mathematics Department, Purdue University, West Lafayette,
IN  47907}
\email{bell@math.purdue.edu}


\begin{abstract}
We characterize the space of restrictions of real
rational functions to certain algebraic Jordan curves
in the plane via the Dirichlet-to-Neumann map associated
to the domain in the complex plane bounded by the curve
and its Bergman kernel. The characterization leads to a
partial fractions-like decomposition for such rational
functions and new ways to describe such Jordan curves.
The multiply connected case is also explored.
\end{abstract}

\maketitle

\theoremstyle{plain}

\newtheorem {thm}{Theorem}[section]
\newtheorem {lem}[thm]{Lemma}

\hyphenation{bi-hol-o-mor-phic}
\hyphenation{hol-o-mor-phic}

\begin{quotation}
\begin{center}
\begin{em}
In honor of L\'aszl\'o's 70th!
\end{em}
\end{center}
\end{quotation}

\section{Introduction}
\label{sec1}

Shopping for a birthday gift for an old friend can be daunting,
but sometimes you stumble onto a gift that seems just right.
L\'aszl\'o Lempert is known for going back to something, and
back again, until he has just the right way of thinking about
it, even in an infinite dimensional setting. I will revisit some
results in \cite{B,B2,B6,BGS}, simplify proofs, and try to
convince the reader that there might be better ways to think
about them.

Suppose that $\gamma$ is a $C^\infty$ smooth real algebraic Jordan
curve in the plane given by the zero set of a polynomial $p(x,y)$
of two real variables. Such a curve can be described locally
by a parameterization of the curve given by a pair of
smooth real analytic functions $(x(t),y(t))$ where $t$ is a
real parameter. By writing $z=x+iy$ and by letting $t$ become
a complex variable $\tau$, a holomorphic map $z(\tau)$ defined
on a neighborhood of an open line segment on the real line that
maps the segment into the curve is obtained. We may assume that
$z'(t)$ is nonvanishing and that $z(\tau)$ locally maps a
neighborhood of the line segment one-to-one onto a collared
neighborhood of the Jordan curve. The function
$$R(w)=z\left(\,\overline{z^{-1}(w)}\,\right),$$
is an antiholomorphic reflection function for the Jordan
curve. Locally, it maps the inside of the curve to the
outside, the outside to inside, and it is its own inverse.
The curve is fixed by the map. The Schwarz function $S(z)$
associated to the curve is the unique holomorphic function
defined on a neighborhood of the curve such that $S(z)=\bar z$
along the curve. It is given by
$$S(z)=\overline{R(z)}.$$

A primary object of study in this paper is the field of
functions that are the restrictions of rational functions
$r(x,y)$ to the Jordan curve. We seek to expand such
functions in a generalized partial fractions expansion.
As is typical in problems in real algebra, we will gain
insight by letting the variables and coefficients be
complex. Indeed, by writing $z=x+iy$, we may express $r(x,y)$
as a quotient $P(z,\bar z)/Q(z,\bar z)$, where $P$ and $Q$ are
polynomials with complex coefficients. The example of the
function $1-z\,\bar z$ on the unit circle alerts us to the
danger that we must avoid the possibility that $Q(z,\bar z)\equiv0$
on $\gamma$. Because the curve is defined by a
polynomial equation, we are in danger of dividing by a
polynomial that vanishes on the curve when we talk of
rational functions of $x$ and $y$ restricted to the
curve. We now define the space of functions that we
intend to study carefully. Keep in mind that that
the space will contain rational functions $r(x,y)=
p(x,y)/q(x,y)$ where $p$ and $q$ are polynomials
with real coefficients, where $q$ is not identically
zero on $\gamma$. That is why the term ``real'' algebraic
geometry is used in the title.

A rational function $R(z,\bar z)$ is given as a
quotient $R(z,\bar z)=P(z,\bar z)/Q(z,\bar z)$ where
$P(z,w)$ and $Q(z,w)$ are relatively prime complex
polynomials (with complex coefficients). For the restriction
of such a rational function to $\gamma$ to be meaningful, we
must assume that $Q(z,\bar z)$ is not identically zero on
the curve. We now consider the functions
$P(z,S(z))$ and $Q(z,S(z))$, which agree with $P(z,\bar z)$
and $Q(z,\bar z)$ on $\gamma$. They are holomorphic on a
collared neighborhood of the curve and their quotient
agrees with the rational function on the curve away
from the zeroes of the denominator. The zeroes of the
denominator are isolated since $Q(z,\bar z)$ is not
identically zero on the curve. This shows that the
rational function $R(z,\bar z)$ is $C^\infty$ smooth on the curve,
except perhaps at finitely many pole-like singular points.
Similarly, the restriction of a rational function that
is not identically zero on the curve has at most finitely
many isolated zeroes on the curve. We will call the space
of such nondegenerate rational functions $R(z,\bar z)$
restricted to the curve in this way $\RR$. It is easy to
verify that $\RR$ is a field.  Let
$\RS$ denote the subspace of $\RR$ of functions without
singularities on $\gamma$. Note that such functions are
$C^\infty$ smooth on $\gamma$.

The purpose of this paper is to show that, after a change
of variables that is close to the identity,
the spaces $\RS$ and $\RR$ can be described nicely
in algebraic terms and that the descriptions give
rise to new ways to represent the curve.
The description also reveals that the algebra of
such rational functions restricted to the curve can
be viewed as a linear space in a manner reminiscent
of the partial fractions decomposition of a rational
function in the plane. The tools used to obtain these
results are the Dirichlet-to-Neumann map associated
to the domain $\Om$ enclosed by $\gamma$ and its
Bergman kernel. The change of variables is given as
a form of generalized Bergman coordinates.

The motivating example for what we are about to do
is the field of restrictions of rational functions
to the unit circle in the plane. The Schwarz function
for the unit circle is $S(z)=1/z$. If $R(z,\bar z)$ is
a function in $\RS$, then $R(z,1/z)$ is a rational
function of $z$ that agrees with $R(z,\bar z)$ on the unit
circle. Expanding $R(z,1/z)$ in partial fractions
yields a decomposition
$$R(z,1/z)=P(z)+\sum_{i=1}^n\sum_{k=1}^{N_n} \frac{A_{ik}}{(z-a_i)^k}+
\sum_{i=1}^m\sum_{k=1}^{M_m} \frac{B_{ik}}{(z-b_i)^k}$$
where $P(z)$ is a polynomial and where the poles $a_i$ are
inside the unit circle and the poles $b_i$ are outside.
(Note that we can think of a polynomial as having a pole
at the point at infinity.) If we now replace $z$ in
the first double sum by $1/\bar z$ and note that the
holomorphic poles of the individual terms are moved from
inside the unit circle to antiholomorphic poles that are
outside of the unit circle by this change, we obtain
a decomposition of $R(z,\bar z)$ on the unit circle as a sum
$r_1(z)+\overline{r_2(z)}$ where $r_1$ and $r_2$ are
rational functions with only poles outside the unit
circle. This decomposition was studied by Peter Ebenfelt
\cite{E} in his studies of the Dirichlet problem with
rational boundary data where he also noted that it
solves the Dirichlet problem with rational boundary data
$R(z,\bar z)$. If $r_1(z)+\overline{r_2(z)}$ were
identically zero on the unit circle, then the Schwarz
reflection principle yields that $r_1(z)$ inside the
unit circle extends $-\overline{r_2(1/\bar z)}$ outside
the unit circle to a bounded entire function. Hence the
decomposition is unique up to the addition of a constant,
and unique if we agree to include that constant in the
$r_1$ term. Replacing $z$ by $1/\bar z$ in the polynomial
and the second double sum instead of the first leads
to a similar decomposition where all the poles of both
rational functions are {\it inside\/} the unit circle.
Thus, we have shown that there are three ways to
represent the restriction of a rational function to
the unit circle: the poles inside decomposition, the
poles outside decomposition, and as the restriction of
a holomorphic rational function to the unit circle.
(A fourth way is as the restriction of an
antiholomorphic rational function to the unit circle,
which can be seen by starting with the conjugate of
the rational function instead and expressing it as
the restriction of a complex rational function.)
We thank the referee for pointing out that Philip
Davis also used related ideas in Chapter~13 of \cite{D}
using the classical method of images of potential
theory on the unit disc.

To accomplish similar decompositions on a more
general algebraic Jordan curve, it will be important
that its Schwarz function extend to the inside of
the domain as a meromorphic function. This
condition is equivalent to the domain being a
quadrature domain with respect to area measure,
as shown by Aharonov and Shapiro \cite{AS} (see
also \cite{S}). Gustafsson \cite{G} later connected
the study of area quadrature domains to the double
of the domain in the multiply connected setting
and we will use many of his ideas in what follows.

Area quadrature domains are dense among bounded
smooth simply connected domains in a very strong
sense, and so our first step is to make a change
of variables that is $C^\infty$ close to the identity
to transform our curve to be the boundary of an area
quadrature domain.

\section{Step 1: A change of variables}
\label{sec2}

Let $\gamma$ and $\Om$ be as in \S\ref{sec1} and let
$K(z,w)$ denote the Bergman kernel associated to $\Om$.
Bergman coordinates that can be taken as close to the
identity as desired are defined in \cite{B6} as
follows. Let $\KK$ denote the Bergman span associate to
$\Om$, which is the complex linear span of all functions of
the form
$$K_a^0(z):=K(z,a)$$
and
$$K_a^m(z):=\left.\frac{\dee^m}{\dee\bar w^m}K(z,w)\right|_{w=a}$$
as $a$ ranges over points in $\Om$ and $m$ ranges over all
nonnegative integers. As shown in \cite{B6}, we may find an
element of $\KK$
as close in $C^\infty(\Obar)$ to the function which is
identically one as desired. If such an element is sufficiently
close to one, it has an antiderivative $f(z)$ that is as close to $z$
in $C^\infty(\Obar)$ as desired such that $f$ is a conformal mapping
of $\Om$ one-to-one onto a nearby domain $\Om_2$. It is shown in
\cite{B6} that such a domain $\Om_2$ is an area quadrature domain,
which is therefore bounded by a smooth real algebraic curve by
results of Aharonov and Shapiro that we will soon review. With this
change of variables, we now focus on smooth real algebraic Jordan
curves that are the boundary of an area quadrature domain.

The history of the study of area quadrature domains in the plane
is long and glorious. The foundations of the subject go back to
those papers of Aharonov and Shapiro \cite{AS} and Gustafsson \cite{G}
(see also \cite{S}, \cite{SU}, and \cite{G2}). Many of the
basic theorems that we will need are proved in \cite[Chap.~22]{B}
using a similar philosophy to the techniques of this paper.

\section{The main results}
\label{sec3}

From now until we consider the multiply connected case, we
will suppose that $\gamma$ is a smooth real algebraic Jordan curve
that bounds an area quadrature domain $\Om$. Then the Schwarz
function $S(z)$ for $\gamma$ extends meromorphically to $\Om$
\cite{AS}. For $z$ in $\gamma$, let $T(z)$ be equal to the
complex unit tangent vector pointing in the direction of the
standard orientation for the boundary $b\Om$ of $\Om$.

The Dirichlet-to-Neumann (D-to-N) map takes a $C^\infty$ smooth
function on $\gamma$ to the normal derivative of its Poisson extension
to $\Om$. The following theorem is proved in \cite[Theorem~1.5]{B2}.

\begin{thm}
\label{thm3.1}
The D-to-N map associated to a bounded smooth simply
connected area quadrature domain $\Om$
maps $\RS$ {\it onto} $\KK T + \overline{\KK\ T}$. It is close
to being one-to-one in the sense that, if two functions map
to the same function, they must differ by a constant.
Furthermore, the elements of the Bergman span in the image
are uniquely determined.
\end{thm}

We remark that, sometimes the D-to-N map is viewed, not as
$\phi$ to harmonic extension $u$ to normal derivative
$\frac{\dee u}{\dee n}$, but as tangential derivative
$\frac{\dee \phi}{\dee s}=\frac{\dee u}{\dee s}$ to
normal derivative $\frac{\dee u}{\dee n}$. In this setting,
the D-to-N map is one-to-one.

We will give a new streamlined proof of Theorem~\ref{thm3.1}
here, but first, some old fashioned mathematical plagiarism.
See \cite{Berg} or \cite{B} for more detailed descriptions of
this quick review of classical formulas.

The Green's function $G(z,a)$ associated to $\Om$ is the
harmonic function of $z$ on $\Om-\{a\}$ that vanishes on the
boundary and is such that $G(z,a)+\ln|z-a|$ has a removable
singularity at $a\in\Om$. The Bergman kernel and the Green's
function are related via
$$K(z,w)=-\frac{2}{\pi}\frac{\dee^2G(z,w)}{\dee z\dee\bar w},$$
and the complementary kernel to the Bergman kernel is defined
via
$$\Lambda(z,w)=-\frac{2}{\pi}\frac{\dee^2G(z,w)}{\dee z\dee w}.$$
The Bergman kernel is $C^\infty$ smooth on
$\Obar\times\Obar$ minus the boundary diagonal and
$\Lambda(z,w)$ is $C^\infty$ smooth on
$\Obar\times\Obar$ minus the diagonal.
The Bergman kernel is holomorphic in $z$ and antiholomorphic
in $w$; $\Lambda(z,w)$ is holomorphic in both variables on
$\Om\times\Om$ minus the diagonal and has a double pole
as a function of $z$ at $z=w$ with principal part
$-\frac{1}{\pi}(z-w)^{-2}$ there. Furthermore, $\Lambda(z,w)=
\Lambda(w,z)$, and
\begin{equation}
\label{KL}
K(z,w)T(z)=-\overline{\Lambda(z,w)}\,\overline{T(z)}
\end{equation}
when $z$ is in the boundary and $w\ne z$.

We will use notation similar to the Bergman span notation
to denote $\Lambda_a^0(z)=\Lambda(z,a)$ and
$$\Lambda_a^m(z)
:=\left.\frac{\dee^m}{\dee w^m}\Lambda(z,w)\right|_{w=a},$$
and denote $G_a^0(z)=G(z,a)$ and
$$G_a^m(z)
:=\left.\frac{\dee^m}{\dee w^m}G(z,w)\right|_{w=a},$$
and
$$G_a^{\overline{m}}(z)
:=\left.\frac{\dee^m}{\dee\bar w^m}G(z,w)\right|_{w=a}.$$
Note that $G_a^m(z)$ has a pole of order $m$ at $z=a$ as
a function of $z$ and that
$\frac{\dee}{\dee\bar z}G_a^m(z)$ is a nonzero complex
constant times the conjugate of $K_a^m(z)$ and
$\frac{\dee}{\dee z}G_a^{\overline{m}}(z)$ is a nonzero complex
constant times $K_a^m(z)$. Similarly,
$\frac{\dee}{\dee z}G_a^m(z)$ is a nonzero complex
constant times $\Lambda_a^m(z)$ and
$\frac{\dee}{\dee\bar z}G_a^{\overline{m}}(z)$ is a nonzero complex
constant times the conjugate of $\Lambda_a^m(z)$.
Note also that differentiating (\ref{KL}) with respect to
$\bar w$ yields that
\begin{equation}
\label{KL_m}
K_w^m(z)T(z)=-\overline{\Lambda_w^m(z)}\,\overline{T(z)}
\end{equation}
when $z$ is in the boundary and $w\ne z$.

For the future, let ${\mathbf\Lambda}$ denote the complex
linear span of the functions $\Lambda_a^m$ as $a$ ranges
over $\Om$ and $m$ ranges over all nonnegative integers.

\begin{proof}[Proof of Theorem~\ref{thm3.1}]
If $u$ is a harmonic function on $\Om$ that
extends smoothly up to an open curve segment on
the boundary, $u$ can be written locally there as
$h+\overline{H}$, where $h$ and $H$ are holomorphic
functions that are smooth up to the segment. The normal
derivative of the holomorphic function $h$ along the
outward pointing normal direction at a point $z$ in
the segment is $-i\,h'(z)T(z)$. Similarly, the normal
derivative of $\overline{H}$ is
$i\,\overline{H'(z)}\,\overline{T(z)}$.
Since $h'(z)=\dee u/\dee z$ and $\overline{H'(z)}=
\dee u/\dee\bar z$, we conclude that the normal
derivative of a harmonic function $u$ is given by
\begin{equation}
\label{normal}
\frac{\dee u}{\dee n}=
-i\,\frac{\dee u}{\dee z}\,T(z)+
i\,\frac{\dee u}{\dee\bar z}\,\overline{T(z)}.
\end{equation}

Given a rational function $R(z,\bar z)$ representing
a function $r(z)$ in $\RS$, let $u(z)$ denote its harmonic
extension to $\Om$. The function $R(z,S(z))$ extends
$r(z)$ to $\Om$ as a meromorphic function and the
harmonic extension $u$ is equal to $R(z,S(z))$ minus
a finite linear combination of derivatives of the
Green's function $G_{a}^{m}(z)$ at points in $a$
in $\Om$ chosen to subtract off the principal parts
of the poles inside $\Om$.
Note that $\frac{\dee u}{\dee\bar z}$ is therefore
seen to be a finite linear combination of functions
$\frac{\dee}{\dee\bar z}G_{a}^{m}(z)$, which is
a linear combination of the conjugates of the functions
$K_{a}^m(z)$. Similarly,
$R(\,\overline{S(z)},\bar z)$ extends $r(z)$ as an
antimeromorphic function in $\Om$ and subtracting off
a linear combination of functions
$G_a^{\overline{m}}(z)$ to eliminate the poles yields
that $\frac{\dee u}{\dee z}$ is a linear combination
of $\frac{\dee}{\dee z}G_a^{\overline{m}}(z)$, which is
a linear combination of $K_a^m(z)$. Now formula
(\ref{normal}) yields the promised decomposition of the
normal derivative of $u$ in terms of the Bergman span.

Note that everything we have done so far is valid in
the multiply connected setting. We now use the
assumption that $\Om$ is simply connected to pin
down the uniqueness of the terms in the expression
involving the Bergman span. If $\kappa_1$ and $\kappa_2$
are two elements of the Bergman span such that
$\kappa_1 T=\overline{\kappa_2 T}$ on the boundary,
then $\kappa_1 dz$ extends to the double via extension
by $\overline{\kappa_2}\, d\bar z$ on the reflected
side. The only holomorphic $1$-form on a Riemann
surface of genus zero is the zero form. Hence
$\kappa_1$ and $\kappa_2$ are both zero on $\Om$.
This is only possible if all the coefficients in
the descriptions of the elements in the Bergman
span are also zero. We will complete the proof when
we show that the D-to-N map takes $\RS$ {\it onto}
the image space, which we will do by means of the
lemma soon to follow.
\end{proof}

In a moment, we will define $k_a^m(z)$ to be a
certain complex antiderivative of $K_a^m(z)$, and
we will see that it has some remarkable properties in
our present context. Let $\Ohat$ denote
the double of $\Om$ and let $\Ot$ denote the reflection
of $\Om$ in $\Ohat$. If $a$ is a point in $\Om$, let
$\tilde a$ denote the point in $\Ot$ represented by
the reflection of $a$ in $\Ohat$.

\begin{lem}
\label{lem3.1}
Suppose $\Om$ is a simply connected smooth area quadrature
domain and $a\in\Om$. A complex antiderivative
$k_a^m(z)$ of $K_a^m(z)$ on $\Om$ extends to the double as
a meromorphic function with a pole of order $m+1$ at the
reflection $\tilde a$ of $a$ in the reflected side of
$\Om$ in the double. Furthermore,
$k_a^m(z)$ is an algebraic function whose restriction
to the boundary of $\Om$ is a rational function of $z$
and $\bar z$ in $\RS$. The normal derivative of
$k_a^m(z)$ is $-iK_a^m(z)T(z)$.
Similarly, the normal derivative of the conjugate of
$k_a^m(z)$ is $i\,\overline{K_a^m(z)T(z)}$.
Consequently, the D-to-N map takes $\RS$ {\it onto}
$\KK T + \overline{\KK\ T}$.
\end{lem}

There are various ways of understanding this lemma,
several of which are explored in \cite{B2} in a
meandering style. Perhaps the best way to explain
it is to note that formula (\ref{KL}) reveals that
the holomorphic $1$-form $K_a^0(z)\,dz$ extends to
the double as a meromorphic $1$-form $\kappa_a^0$
via extension by the conjugate of
$-\Lambda_a^0(z)\,dz$ on the reflected side.
This $1$-form has a single residue free double pole at
the point $\tilde a$ on the reflected side $\Ot$. If
we fix a point $b$ in $\Om$, then the line integral
$$h(z)=\int_{\gamma_b^z}\kappa_a^0$$
along a curve $\gamma_b^z$ connecting $b$ to
a point $z$ in the double that avoids $\tilde a$
is well defined and defines a single valued
holomorphic function on the double minus $\tilde a$
with a simple pole at~$\tilde a$. Indeed, $\kappa_a^0=dh$.
We define the function $k_a^0$ to be the restriction
of $h$ to $\Obar$. (There will be occasion to think of
$k_a^0$ as equal to the meromorphic function $h$ on $\Ohat$
later in the paper.)

Similarly, (\ref{KL_m}) reveals that
the holomorphic $1$-form $K_a^m(z)\,dz$ extends to
the double as a meromorphic $1$-form $\kappa_a^m$
via extension by the conjugate of
$-\Lambda_a^m(z)\,dz$ on the reflected side.
This $1$-form has a single residue free pole of
order $m+2$ at $\tilde a$ in the reflected side.
Therefore the line integral
\begin{equation}
\label{kam}
h(z)=\int_{\gamma_b^z}\kappa_a^m
\end{equation}
along a curve $\gamma_b^z$ as above is well defined and
defines a single valued holomorphic function on the double
minus $\tilde a$ with a pole of order $m+1$ at $\tilde a$.
Let $k_a^m$ be the restriction of this function to $\Obar$.
(As with $k_a^0$, we will later consider $k_a^m$ to be
given by the meromorphic function $h$ on $\Ohat$.)

We note that we have defined the functions $k_a^m$ on
$\Om$ via
$$k_a^m(z)=
\int_{\gamma_b^z}K_a^m(w)\ dw,$$
and $k_a^m$ are therefore seen to be homorphic in $z$ and
antiholomorphic in $a$ on $\Om$, and normalized by the
property that they all vanish at $b$.

That restrictions to the boundary of $\Om$ of meromorphic
functions on the double of an area quadrature domain are
rational functions of $z$ and $\bar z$ was shown by
Bj\"orn Gustafsson in \cite{G}. We describe his argument
here. Since $S(z)=\bar z$ on the boundary,
the function $S(z)$ extends to the double of $\Om$ as
a meromorphic function $G_1$ via extension by the function
$\bar z$ on the reflected side. The conjugate identity,
$z=\overline{S(z)}$ on $b\Om$, yields that the function
$z$ extends to the double as a meromorphic function $G_2$
via extension by the function $\overline{S(z)}$ on the
reflected side. By taking a point $p$ sufficiently close
to the point at infinity in the Riemann sphere, we may
arrange that $G_1^{-1}(p)$ consists of distinct points
of multiplicity one inside $\Om$. Since $G_2(z)$ separates
these points, $G_1$ and $G_2$ form a primitive pair for
the double, and therefore generate the field of meromorphic
functions (see Farkas and Kra \cite[p.~249]{FK}). Hence,
meromorphic functions on the double are rational
combinations $G_1$ and $G_2$, and hence, functions of
$z$ and $\bar z$ when restricted to the boundary of
$\Om$. Primitive pairs are always algebraicly
dependent, and hence there is an irreducible polynomial
$P(z,w)$ such that $P(z,S(z))\equiv0$ on $\Om$. Hence,
$S(z)$ is an algebraic function and all meromorphic
functions on $\Om$ that extend meromorphically to the
double are algebraic. It also follows that $b\Om$ is
in the set where $P(z,\bar z)=0$, and Gustafsson explores
the nature of this algebraic curve further in \cite{G}.
This completes the proof of the lemma.

Let $\mathbf k$ denote the complex linear span of
all the functions $k_a^m$ defined by the integral
(\ref{kam}) as $a$ ranges over points in $\Om$ and
$m$ ranges over all nonnegative integers.

The following theorem now follows by noting that the
normal derivative of $k_a^m$ is $-iK_a^mT$ and that
two functions on the boundary that map to the same
function via the D-to-N map must differ by a constant.

\begin{thm}
\label{thm3.2}
On a smooth simply connected area quadrature domain
$\Om$, a rational function $R(z,\bar z)$ in $\RS$ can be
decomposed as a constant plus a uniquely determined
element of $\mathbf k+\overline{\mathbf k}$. Thus,
$$\RS=\C+\mathbf k+\overline{\mathbf k}.$$
\end{thm}

It is interesting to note that the complex polynomials
are in the Bergman span associated to an area quadrature
domain (see \cite[p.~119]{B}). Hence, we deduce that 
$\mathbf k$ contains all antiderivatives of complex
polynomials that vanish at $b$, i.e., all polynomials that
vanish at $b$. Furthermore, it can be shown that
all complex rational functions with residue free poles
outside $\Obar$ that vanish at $b$ are also in $\mathbf k$.

With this theorem behind us, we can compare it to
the decomposition of rational functions of $z$ and $\bar z$
restricted to the unit circle we described as our motivator
early on and see that they are one and the same result!

Another pleasant moment arises when we realize that
Theorem~(\ref{thm3.2}) gives a solution to the Dirichlet
problem in terms of the Bergman kernel on a smooth simply
connected area quadrature domain when the boundary data
is real and rational. A silly way to solve the Dirichlet
problem would be to solve it with real rational data,
appeal to the fact that real rational functions on
the boundary are an algebra that separates points and
are therefore dense among continuous functions, then
use sup norm estimates to get general solutions
from special solutions. Going back to such foundational
considerations is a hobby of the author. See \cite{BR}
for what happens when one allows oneself to get carried away
with such thoughts.

We remark here that all the arguments we have used
up to this point can be repeated using the kernel
$\Lambda(z,w)$ in place of $K(z,w)$ to yield the
following theorems. Recall that $\mathbf\Lambda$ denotes
the complex linear span of the functions $\Lambda_a^m$ as
$a$ ranges over $\Om$ and $m$ ranges over all nonnegative
integers.

\begin{thm}
\label{thm3.3}
The D-to-N map associated to a smooth area quadrature domain
maps $\RS$ {\it onto}
${\mathbf\Lambda} T + \overline{{\mathbf\Lambda}\ T}$. It is close
to being one-to-one in the sense that, if two functions map
to the same function, they must differ by a constant.
Furthermore, the elements of the span $\mathbf\Lambda$ of the
complementary kernel to the Bergman kernel in the image are
uniquely determined.
\end{thm}

Actually, Theorem~(\ref{thm3.3}) could also be deduced from
Theorem~(\ref{thm3.1}) by merely applying the identity
(\ref{KL_m}) to the terms in the decomposition.

The same reasoning that led to the definitions of the
antiderivatives $k_a^m$ yield that antiderivatives of
$\Lambda_a^m$ exist and have special properties. Indeed,
let $\eta_a^m$ be the
meromorphic $1$-form on $\Ohat$ defined as
$\Lambda_a^m(z)\, dz$ on $\Om$ and as
$-\overline{K_a^m(z)}\,d\bar z$ on $\Ot$, and define
\begin{equation}
\label{lam}
h(z)=\int_{\gamma_{\tilde b}^z}\eta_a^m,
\end{equation}
where $\tilde b$ is the reflection of the point $b$
used in equation (\ref{kam}) and $\gamma_{\tilde b}^z$ is
a curve starting at $\tilde b$ in $\Ohat$ and ending at
$z$ which avoids the point $a$. Define $\lambda_a^m$ to
be the restriction of $h$ to $\Obar$ and let
$\blambda$ denote the complex linear span of $\lambda_a^m$
as $a$ ranges over $\Om$ and $m$ ranges over all
nonnegative integers. The next theorem follows by the
same reasoning we applied to the $k_a^m$.

\begin{thm}
\label{thm3.4}
On a smooth area quadrature domain $\Om$,
a rational function $R(z,\bar z)$ in $\RS$ can be
decomposed as a constant plus a uniquely determined
element of
\boldmath
$$\lambda+\overline{\lambda}.$$
\unboldmath
\end{thm}

It is interesting to note that, by the way we have
defined them via (\ref{kam}) and (\ref{lam}),
$k_a^m(z)=-\overline{\lambda_a^m(z)}$ on the boundary,
and this is another way to see that both functions
extend meromorphically to the double. Also, the
argument principle applied to the boundary relation
yields that the one zero of $k_a^0$ at $b$
balances the one pole of $\lambda_a^0$ at
$a$ and, because zeroes of $k_a^0$ get counted
positively on the left and zeroes of $\lambda_a^0$
count negatively on right, there can be no other zeroes
of $k_a^0$ in $\Obar$ and no zeroes of
$\lambda_a^0$ in $\Obar$. Consequently,
$k_a^0/\lambda_a^0$ is seen to be equal to a
unimodular constant times the product of the Riemann
map that maps $b$ to the origin and the Riemann map that
maps $a$ to the origin, and $k_b^0/\lambda_b^0$ is a
unimodular constant times the square of the Riemann
map that maps $b$ to the origin. It is not hard to use
the transformation formula for the Bergman kernel under
conformal mappings and the formula for the Bergman kernel
of the unit disc to see that
$$k_a^0(z)=\frac{f(z)\overline{f'(a)}}{\pi(1-f(z)\,\overline{f(a)})},$$
where $f$ is the Riemann map mapping $\Om$ one-to-one onto
the unit disc with $f(b)=0$ and $f'(b)>0$. Similarly,
$$\lambda_a^0(z)=\frac{f'(a)}{\pi(f(a)-f(z))}.$$

Theorem~\ref{thm3.3} and~\ref{thm3.4} are to the previous
two as choosing poles inside the unit circle is to choosing
them outside in the motivator decomposition of restrictions
of rational function to the unit disc.

Lemma~(\ref{lem3.1}) also carries over naturally.

\begin{lem}
\label{lem3.2}
On a simply connected smooth area quadrature domain
$\Om$, a complex antiderivative
$\lambda_a^m(z)$ of $\Lambda_a^m(z)$ on $\Om$ extends to the double as
a meromorphic function. It is holomorphic on the reflected
side and has a pole of order $m+1$ at $z=a$.
Furthermore,
$\lambda_a^m(z)$ is an algebraic function whose restriction
to the boundary of $\Om$ is a rational function of $z$
and $\bar z$. Finally, the normal derivative of
$\lambda_a^m(z)$ is $-i\Lambda_a^m(z)T(z)$ and so the D-to-N map
takes $\RS$ {\it onto} ${\mathbf\Lambda} T +
\overline{{\mathbf\Lambda}\ T}$.
\end{lem}

Finally, notice that, by identity (\ref{KL_m}), the terms
$\overline{K_a^m T}$ in Theorem~\ref{thm3.1} can be
rewritten as $-\Lambda_a^m T$, and following the rest
of the argument through in the proof of Theorem~\ref{thm3.2}
yields that
$$\RS=\C+{\mathbf k}+\blambda.$$
Note that a rational function $R(z,\bar z)$
extends to be a meromorphic function on the double via
the formula $R(z,S(z))$, since $S(z)$ extends to the
double. The elements ${\mathbf k}+\blambda$ in the
decomposition are now seen to be completely determined
by the principal parts of the poles of $R(z,S(z))$ in
$\Om$ and the reflection of $\Om$. The poles inside
determine the $\blambda$ element,
and the poles on the reflected side determine the
$\mathbf k$ element. Hence, it is a genuine analogue
of a partial fractions decomposition. It corresponds
to the simple observation that the restrictions of
rational functions to the unit circle are the same
as the restrictions of complex rational functions.
This is the point where it would be more natural to
define the functions $k_a^m$ and $\lambda_a^m$ as
the meromorphic functions on the double to which they
extend. Then the formula for $R(z,\bar z)$ on
the boundary extends to a formula expressing $R(z,S(z))$
on the double and it is clear how the poles determine
the elements in the expansion.

Finally, we remark that all the results of this paper extend
to allow pole-like singularities of rational functions
restricted to the boundary. Indeed, formulas (\ref{KL})
and (\ref{KL_m}) are valid when both $z$ and $w$ are in the
boundary $z\ne w$ because the Green's function identities
also extend to $b\Om\times b\Om$ minus the boundary
diagonal. Hence, $\RR$ can be decomposed as a vector space
sum of the antiderivatives of elements in the extended
Bergman span (which include base points $a$ on the boundary),
and conjugates of such functions. It is interesting to
note that
$$T(z)K(z,w)\,\overline{T(w)}=-\overline{T(z)\Lambda(z,w)T(w)}
=T(w)K(w,z)\ \overline{T(z)},
$$
when both $z$ and $w$ are on the boundary, and so the terms
involving $\overline{K_w^m(z)}\,\overline{T(z)}$ can be
converted to terms $K_w^m(z)T(z)$ when $w$ is in the boundary.

\section{Special descriptions of the boundary of an area
quadrature domain}
\label{sec4}

We are now in a position to apply our results to describing
the boundary curves of a simply connected smooth area
quadrature domain. If $a$ is a point on the boundary, the
function
$$\frac{1}{S(z)-\bar a}$$
has a simple pole at $a$, and so there is a constant $A$
such that
$$R(z)=\frac{1}{S(z)-\bar a}-\frac{A}{z-a}$$
has a removable singularity at $a$. Note that, because
the boundary of $\Om$ cannot be a line, $R(z)$ is a
function in $\RS$ that is not the zero function, and
so can be expressed as a sum $c+k_1+\overline{k_2}$ where
$c$ is a constant and $k_1$ and $k_2$ are elements in
$\mathbf k$ that are not both zero. Hence, the Jordan
curve is given by the set where
$$\frac{1}{\bar z-\bar a}-
\overline{k_2(z)}
=
\frac{A}{z-a}+c+
k_1(z),$$
where $k_1$ and $k_2$ are algebraic functions that are
holomorphic on $\Om$ whose smooth boundary values are
rational functions of $z$ and $\bar z$. Using the other
decompositions lead to interesting variations, which we
leave to the reader.

\section{Back to step~1}
\label{sec5}

The change of variables introduced in \S\ref{sec2}
has some interesting properties that we now discuss.

The main theorems of this paper can be thought of as
being invariant under changes of Bergman coordinates. Indeed,
as shown in \cite{B6}, a biholomorphic mapping between
area quadrature domains is an algebraic function whose
restriction to the boundary is rational in $z$ and
$\bar  z$ because it and its inverse extend to the doubles
as meromorphic functions. Hence, the mapping is a
birational map when restricted to the boundaries which
preserves all the algebraic properties of the decompositions
described in \S\ref{sec3}.

Area quadrature domains are such that elements of
their Bergman span are algebraic functions whose
restrictions to the boundary are rational in $z$
and $\bar z$ because they extend to the double as
meromorphic functions. Furthermore, the function
$T(z)$ on the boundary is such that $T(z)^2$ extends
to the double as a meromorphic function, and $T(z)^2$
is therefore a rational function of $z$ and $\bar z$
on the boundary. Consequently the D-to-N map on
an area quadrature domain takes rational functions
of $z$ and $\bar z$ to real algebraic functions.

The procedure described in Step~1 to find Bergman
coordinates close to the identity can easily be
adapted to prove the same result in the multiply
connected setting. It is worth noting here that
there is a very elementary procedure for producing
the change of variables in the simply connected
case, as shown in \cite{BGS}. Indeed, suppose $\Om$
is a bounded domain
with smooth real analytic boundary. Then a Riemann
mapping function $F:\Om\to D_1(0)$ extends holomorphically
past the boundary. Approximate the inverse $F^{-1}$
by a polynomial $P(z)$ on a neighborhood of the closed
unit disc. Shapiro and Aharonov \cite{AS} proved
that simply connected area quadrature domains
are given as images of the unit disc under
rational conformal mappings. Hence the image of
the unit disc under $P$ is an area quadrature
domain nearby to $\Om$ and the change of variables is
$P\circ F$. Shapiro and Ullemar \cite{SU} proved
that a simply connected domain $\Om$ is a quadrature
domain with respect to arclength if and only if the
derivative of the inverse of the Riemann map is the
square of a rational function. We can approximate
a square root of the nonvanishing $(F^{-1})'$  by a
polynomial $p(z)$, and let $P(z)$ be a polynomial
antiderivative of $p(z)^2$ to obtain a domain
nearby to $\Om$ given by the image of the unit disc
under $P$ that is a quadrature domain with
respect to both area and arclength. This domain
has all the properties we admire in an area
quadrature domain plus the property that $T(z)$
extends to the double as a meromorphic function.
Hence, it is a rational function of $z$ and $\bar z$
on the boundary and we have a domain whose
D-to-N map takes rational functions of $z$ and $\bar z$
to rational functions of $z$ and $\bar z$.

\section{The multiply connected case}
\label{sec6}

Assume that $\Om$ is a bounded $n$-connected smooth area
quadrature domain. Many of the same arguments we have
used in the simply connected setting can be carried
through without change, and others with only minor
modifications involving the harmonic measure functions
$\omega_j$ and the associated holomorphic functions $F_j'$,
$j=1,\dots,n-1$. We sketch the main arguments here.

Let $\gamma_n$ denote the outer boundary
curve of $\Om$ and let $\gamma_j$, $j=1,\dots,n-1$
denote the $n-1$ inner boundary curves. The function
$\omega_j$ is the harmonic function on $\Om$ with
boundary values equal to one on $\gamma_j$ and zero
on the other boundary curves. The function $F_j'$
is a holomorphic function on $\Om$ equal to
$2(\dee\omega_j/\dee z)$. If one considers a locally
defined harmonic conjugate function $v_j$ for
$\omega_j$, then $F_j'$ is the complex derivative
of $\omega_j+iv_j$. Even though $v_j$ is only locally
defined, $F_j'$ is globally defined. However, even
though we use a prime in the notation, $F_j'$ is
not the complex derivative of a function defined
on $\Om$ when $n>1$.

It is well know that the matrix of periods of
the $F_j'$ on the $n-1$ inner boundary curves is
nonsingular. Hence, there are unique constants
$c_j$ and a unique holomorphic function $k_a^m(z)$
on $\Om$ such that
\begin{equation}
\label{kKF}
(k_a^m)'(z)=K_a^m(z)+\sum_{j=1}^{n-1}c_jF_j'(z).
\end{equation}
The functions $k_a^m$ can be defined by line
integrals as in \S\ref{sec3}. Let $\mathbf k$ denote
the complex linear span of the functions $k_a^m$.

Because $\omega_j$ is constant on boundary curves,
its tangential derivative
$$\frac{\dee\omega_j}{\dee z}\ T(z)+
\frac{\dee\omega_j}{\dee\bar z}\ \overline{T(z)}$$
is zero on the boundary, and we deduce the well-known
identity
$$F_j'(z)\,dz=-\overline{F_j'(z)}\, d\bar z$$
on the boundary. Furthermore, the normal derivative
of $\omega_j$ is given by
$$-i
\frac{\dee\omega_j}{\dee z}\ T(z)+i
\frac{\dee\omega_j}{\dee\bar z}\ \overline{T(z)}
=-2i \frac{\dee\omega_j}{\dee z}\ T(z)=-iF_j'T(z).$$

The part of the new proof we gave of Theorem~\ref{thm3.1} in
the simply connected case that showed that the image of
the rational functions $\RS$ under the D-to-N map is equal
to $\KK T+\overline{\KK T}$ is valid in the
multiply connected setting.

\begin{thm}
\label{thm6.1}
The D-to-N map associated to a bounded smooth finitely
connected area quadrature domain $\Om$
maps $\RS$ into $\KK T + \overline{\KK\ T}$. It is close
to being one-to-one in the sense that, if two functions map
to the same function, they must differ by a constant.
\end{thm}

Given a rational function
$R(z,\bar z)$ in $\RS$, let $\phi$ denote its harmonic
extension to $\Om$ and represent the normal derivative
$$\frac{\dee\phi}{\dee n}=\kappa_1 T+\overline{\kappa_2 T},$$
where $\kappa_1$ and $\kappa_2$ are in the Bergman span.
Since the normal derivative of $k_a^m$ is $-i(k_a^m)'T$,
formula (\ref{kKF}) yields that there are elements $k_1$
and $k_2$ in $\mathbf k$ and constants $c_j$ such that
the normal derivative of
$$k_1+\overline{k_2}+\sum_{j=1}^{n-1}c_j\omega_j$$
is equal to the normal derivative of $\phi$. Hence,
$\phi$ and this expression differ by a constant and
we conclude that $R(z,\bar z)$ and $k_1+\overline{k_2}$
differ by a constant on each boundary component. Thus,
the same result we obtained in the simply connected case
holds in the multiply connected setting, with that added
stipulation that the constants involved might be
different on different boundary curves. We can also
use the Bergman kernel to solve the Dirichlet problem
with real rational boundary data if we are willing
to toss in the functions $\omega_j$. These ideas are
explored further in \cite{B2}

\section{Similar results on nonquadrature domains}
\label{sec7}

We close by remarking that many of the results of this
paper extend naturally to the setting of general
bounded domains with smooth boundary if we replace
the space of rational functions of $z$ and $\bar z$
restricted to the boundary by the space of functions
on the boundary that extend meromorphically to the
double of the domain, as shown in \cite{B2}.

\end{document}